\documentclass[12pt]{article}
\PassOptionsToPackage{hyphens}{url}
\usepackage[pagebackref=true,colorlinks]{hyperref}
\usepackage{amsmath,amssymb,amsthm,mathtools,amsrefs,mathrsfs}
\usepackage{pifont}
\usepackage{mathpazo} 
\usepackage{eulervm}
\usepackage{bbm}
\usepackage{csquotes}
\usepackage{lscape}
\usepackage{etex} 
\usepackage[table]{xcolor}
\usepackage{tikz}
\usetikzlibrary{arrows.meta}
\usetikzlibrary{decorations.pathmorphing,shapes}
\usetikzlibrary{calc}
\usepackage{comment}
\usepackage{adjustbox}
\usepackage{scalerel,stackengine}
\usepackage{stmaryrd}
\usepackage{fancyhdr}
\usepackage{soul}
\usepackage{dsfont}
\usepackage[normalem]{ulem}
\usepackage{longtable}
\usepackage[bottom]{footmisc}
\usepackage{caption}
\usepackage{subfig} 
\usepackage{pictexwd, dcpic}
\usepackage{float}
\usepackage{enumerate}
\usepackage[all]{xy} 
\usetikzlibrary{circuits.ee.IEC}
\hypersetup{
	colorlinks=true,
    linktoc=all,
    linkcolor=blue,  
    citecolor=blue,
    }

\makeatletter
\newcommand*{\shifttext}[2]{%
  \settowidth{\@tempdima}{#2}%
  \makebox[\@tempdima]{\hspace*{#1}#2}%
}
\makeatother
\makeatletter
\renewcommand*\env@matrix[1][\arraystretch]{%
  \edef\arraystretch{#1}%
  \hskip -\arraycolsep
  \let\@ifnextchar\new@ifnextchar
  \array{*\c@MaxMatrixCols c}}
\makeatother

\stackMath
\newcommand\reallywidehat[1]{%
\savestack{\tmpbox}{\stretchto{%
  \scaleto{%
    \scalerel*[\widthof{\ensuremath{#1}}]{\kern.1pt\mathchar"0362\kern.1pt}%
    {\rule{0ex}{\textheight}}
  }{\textheight}%
}{2.4ex}}%
\stackon[-6.9pt]{#1}{\tmpbox}%
}
\makeatletter
\pgfmathdeclarefunction{erf}{1}{%
  \begingroup
    \pgfmathparse{#1 > 0 ? 1 : -1}%
    \edef\sign{\pgfmathresult}%
    \pgfmathparse{abs(#1)}%
    \edef\x{\pgfmathresult}%
    \pgfmathparse{1/(1+0.3275911*\x)}%
    \edef\t{\pgfmathresult}%
    \pgfmathparse{%
      1 - (((((1.061405429*\t -1.453152027)*\t) + 1.421413741)*\t 
      -0.284496736)*\t + 0.254829592)*\t*exp(-(\x*\x))}%
    \edef\y{\pgfmathresult}%
    \pgfmathparse{(\sign)*\y}%
    \pgfmath@smuggleone\pgfmathresult%
  \endgroup
}
\makeatother

\theoremstyle{theorem}
\newtheorem{theorem}[equation]{Theorem}
\newtheorem{lemma}[equation]{Lemma}
\newtheorem{proposition}[equation]{Proposition}
\newtheorem{corollary}[equation]{Corollary}
\theoremstyle{definition}
\newtheorem{definition}[equation]{Definition}
\newtheorem{construction}[equation]{Construction}
\newtheorem{question}[equation]{Question}
\newtheorem{problem}[equation]{Problem}
\newtheorem{example}[equation]{Example}
\newtheorem{exercise}[equation]{Exercise}
\newtheorem*{answer}{Answer}
\newtheorem*{solution}{Solution}
\newtheorem{remark}[equation]{Remark}

\newtheorem{notation}[equation]{Notation}
\newtheorem{noterm}[equation]{Notation and Terminology}

\newcommand\define[1]{\emph{\textbf{#1}}}

\numberwithin{equation}{section}

 \let\t=\tau

\newcommand{\be}{\begin{equation}}
\newcommand{\ee}{\end{equation}}
\def\ba{\begin{align}} 
\def\ea{\end{align}}
\newcommand{\bea}{\begin{eqnarray}}
\newcommand{\eea}{\end{eqnarray}}
\newcommand{\bx}{\begin{example}}
\newcommand{\ex}{\end{example}}
\newcommand{\bex}{\begin{exercise}}
\newcommand{\eex}{\end{exercise}}
\newcommand{\ban}{\begin{answer}}
\newcommand{\ean}{\end{answer}}
\newcommand{\bt}{\begin{theorem}}
\newcommand{\et}{\end{theorem}}
\newcommand{\bc}{\begin{corollary}}
\newcommand{\ec}{\end{corollary}}
\newcommand{\blem}{\begin{lemma}}
\newcommand{\elem}{\end{lemma}}
\newcommand{\bp}{\begin{problem}}
\newcommand{\ep}{\end{problem}}
\newcommand{\bn}{\begin{proposition}}
\newcommand{\en}{\end{proposition}}
\newcommand{\bd}{\begin{definition}}
\newcommand{\ed}{\end{definition}}
\newcommand{\bcon}{\begin{construction}}
\newcommand{\econ}{\end{construction}}
\newcommand{\bq}{\begin{question}}
\newcommand{\eq}{\end{question}}
\newcommand{\bprf}{\begin{proof}}
\newcommand{\eprf}{\end{proof}}
\newcommand{\br}{\begin{remark}}
\newcommand{\er}{\end{remark}}
\newcommand{\bs}{\begin{solution}}
\newcommand{\es}{\end{solution}}
\newcommand{\beqs}{\begin{eqnarray}}
\newcommand{\eeqs}{\end{eqnarray}}
\newcommand{\bnt}{\begin{noterm}}
\newcommand{\ent}{\end{noterm}}
\newcommand{\bnot}{\begin{notation}}
\newcommand{\enot}{\end{notation}}

\def\R{{{\mathbb R}}}

\def\N{{{\mathbb N}}}
\def\Z{{{\mathbb Z}}}

\def\B{{{\mathbb B}}}

\def\bE{{{\mathbb E}}}

\newcommand{\stoch}{\;\xy0;/r.25pc/:(-3,0)*{}="1";(3,0)*{}="2";{\ar@{~>}"1";"2"|(1.06){\hole}};\endxy\!}

\newcounter{sarrow}

\newcounter{sqarrow}

\newcommand{\ben}{\renewcommand{\theenumi}{\alph{enumi}} 
\renewcommand{\labelenumi}{(\theenumi)}\begin{enumerate}}
\newcommand{\een}{\end{enumerate}}
\newlength\stateheight
\newlength\minimumstatewidth
\tikzset{width/.initial=\minimummorphismwidth}
\tikzset{colour/.initial=white}

\newif\ifblack\pgfkeys{/tikz/black/.is if=black}
\newif\ifwedge\pgfkeys{/tikz/wedge/.is if=wedge}
\newif\ifvflip\pgfkeys{/tikz/vflip/.is if=vflip}
\newif\ifhflip\pgfkeys{/tikz/hflip/.is if=hflip}
\newif\ifhvflip\pgfkeys{/tikz/hvflip/.is if=hvflip}

\pgfdeclarelayer{foreground}
\pgfdeclarelayer{background}
\pgfsetlayers{background,main,foreground}

\def\thickness{0.4pt}

\makeatletter
\pgfkeys{%
  /tikz/on layer/.code={
    \pgfonlayer{#1}\begingroup
    \aftergroup\endpgfonlayer
    \aftergroup\endgroup
  },
  /tikz/node on layer/.code={
    \gdef\node@@on@layer{%
      \setbox\tikz@tempbox=\hbox\bgroup\pgfonlayer{#1}\unhbox\tikz@tempbox\endpgfonlayer\pgfsetlinewidth{\thickness}\egroup}
    \aftergroup\node@on@layer
  },
  /tikz/end node on layer/.code={
    \endpgfonlayer\endgroup\endgroup
  }
}
\def\node@on@layer{\aftergroup\node@@on@layer}
\makeatother

\makeatletter
\pgfdeclareshape{state}
{
    \savedanchor\centerpoint
    {
        \pgf@x=0pt
        \pgf@y=0pt
    }
    \anchor{center}{\centerpoint}
    \anchorborder{\centerpoint}
    \saveddimen\overallwidth
    {
        \pgf@x=3\wd\pgfnodeparttextbox
        \ifdim\pgf@x<\minimumstatewidth
            \pgf@x=\minimumstatewidth
        \fi
    }
    \savedanchor{\upperrightcorner}
    {
        \pgf@x=.5\wd\pgfnodeparttextbox
        \pgf@y=.5\ht\pgfnodeparttextbox
        \advance\pgf@y by -.5\dp\pgfnodeparttextbox
    }
    \anchor{A}
    {
        \pgf@x=-\overallwidth
        \divide\pgf@x by 4
        \pgf@y=0pt
    }
    \anchor{B}
    {
        \pgf@x=\overallwidth
        \divide\pgf@x by 4
        \pgf@y=0pt
    }
    \anchor{text}
    {
        \upperrightcorner
        \pgf@x=-\pgf@x
        \ifhflip
            \pgf@y=-\pgf@y
            \advance\pgf@y by 0.4\stateheight
        \else
            \pgf@y=-\pgf@y
            \advance\pgf@y by -0.4\stateheight
        \fi
    }
    \backgroundpath
    {
       \begin{pgfonlayer}{foreground}
        \pgfsetstrokecolor{black}
        \pgfsetlinewidth{\thickness}
        \pgfpathmoveto{\pgfpoint{-0.5*\overallwidth}{0}}
        \pgfpathlineto{\pgfpoint{0.5*\overallwidth}{0}}
        \ifhflip
            \pgfpathlineto{\pgfpoint{0}{\stateheight}}
        \else
            \pgfpathlineto{\pgfpoint{0}{-\stateheight}}
        \fi
        \pgfpathclose
        \ifblack
            \pgfsetfillcolor{black!50}
            \pgfusepath{fill,stroke}
        \else
            \pgfusepath{stroke}
        \fi
       \end{pgfonlayer}
    }
}

\pgfdeclareshape{my ground}
{
  \inheritsavedanchors[from=rectangle ee]
  \inheritanchor[from=rectangle ee]{center}
  \inheritanchor[from=rectangle ee]{north}
  \inheritanchor[from=rectangle ee]{south}
  \inheritanchor[from=rectangle ee]{east}
  \inheritanchor[from=rectangle ee]{west}
  \inheritanchor[from=rectangle ee]{north east}
  \inheritanchor[from=rectangle ee]{north west}
  \inheritanchor[from=rectangle ee]{south east}
  \inheritanchor[from=rectangle ee]{south west}
  \inheritanchor[from=rectangle ee]{input}
  \inheritanchor[from=rectangle ee]{output}
  \anchorborder{\csname pgf@anchor@my ground@input\endcsname}

  \backgroundpath{
    \pgf@process{\pgfpointadd{\southwest}{\pgfpoint{\pgfkeysvalueof{/pgf/outer xsep}}{\pgfkeysvalueof{/pgf/outer ysep}}}}
    \pgf@xa=\pgf@x \pgf@ya=\pgf@y
    \pgf@process{\pgfpointadd{\northeast}{\pgfpointscale{-1}{\pgfpoint{\pgfkeysvalueof{/pgf/outer xsep}}{\pgfkeysvalueof{/pgf/outer ysep}}}}}
    \pgf@xb=\pgf@x \pgf@yb=\pgf@y
    \pgfpathmoveto{\pgfqpoint{\pgf@xa}{\pgf@ya}}
    \pgfpathlineto{\pgfqpoint{\pgf@xa}{\pgf@yb}}
    \pgfmathsetlength\pgf@xa{.5\pgf@xa+.5\pgf@xb}
    \pgfmathsetlength\pgf@yc{.16666\pgf@yb-.16666\pgf@ya}
    \advance\pgf@ya by\pgf@yc
    \advance\pgf@yb by-\pgf@yc
    \pgfpathmoveto{\pgfqpoint{\pgf@xa}{\pgf@ya}}
    \pgfpathlineto{\pgfqpoint{\pgf@xa}{\pgf@yb}}
    \advance\pgf@ya by\pgf@yc
    \advance\pgf@yb by-\pgf@yc
    \pgfpathmoveto{\pgfqpoint{\pgf@xb}{\pgf@ya}}
    \pgfpathlineto{\pgfqpoint{\pgf@xb}{\pgf@yb}}
  }
}
\makeatother

\tikzset{inline text/.style =
  {text height=1.2ex,text depth=0.25ex,yshift=0.5mm}}
\tikzset{arrow box/.style =
  {rectangle,inline text,fill=white,draw,
    minimum height=5mm,yshift=-0.5mm,minimum width=5mm}}
\tikzset{bubble/.style =
  {inner sep=0mm,minimum width=3mm,minimum height=3mm,
    draw,shape=circle,fill=white}}
\tikzset{dot/.style =
  {inner sep=0mm,minimum width=1mm,minimum height=1mm,
    draw,shape=circle}}
\tikzset{white dot/.style = {dot,fill=white,text depth=-0.2mm}}
\tikzset{scalar/.style = {diamond,draw,inner sep=1pt}}
\tikzset{square/.style =
  {inner sep=0mm,minimum width=2mm,minimum height=2mm,
    draw,shape=rectangle}}

\tikzset{star/.style = {dot,fill=white,text depth=-0.2mm}}
\tikzset{copier/.style = {dot,fill,text depth=-0.2mm}}
\tikzset{fakecopier/.style = {square,fill,text depth=-0.2mm}}
\tikzset{discarder/.style = {my ground,draw,inner sep=0pt,
    minimum width=4.2pt,minimum height=11.2pt,anchor=input,rotate=90}}

\tikzset{xshiftu/.style = {shift = {(#1, 0)}}}
\tikzset{yshiftu/.style = {shift = {(0, #1)}}}
\tikzset{scriptstyle/.style={font=\everymath\expandafter{\the\everymath\scriptstyle}}}

\begin{document}

\begin{center}{\Large On analytic groupoid cardinality}\\
James Fullwood \end{center}

\begin{abstract}
Groupoids graded by the groupoid of bijections between finite sets admit generating functions which encode the groupoid cardinalities of their graded components. As suggested in the work of Baez and Dolan, we use analytic continuation of such generating functions to define a complex-valued cardinality for groupoids whose usual groupoid cardinality diverges. The complex nature of such a cardinality invariant is shown to reflect a recursion of structure which we refer to as `nested equivalence'. 
\end{abstract}

\tableofcontents

\section{Introduction}

Groupoid cardinality was introduced by Baez and Dolan as a way of catgegorifying division of natural numbers \cite{BAEZDOLAN}. In particular, if a finite group $G$ acts freely on a finite set $X$, then the quotient $X/G$ satisfies $\#\left(X/G\right)=\#X/\#G$, but what if the action of $G$ on $X$ is not free? In such a case, one may consider the action groupoid $X/\!\!/ G$, whose objects are the elements of $X$, together with a morphism $g:x_1\to x_2$ whenever $gx_1=x_2$ for some $g\in G$. The groupoid cardinality of $X/\!\!/ G$ is then the rational number $\#X/\#G$, and is a measure of the size of a discrete groupoid which extends the notion of set cardinality when viewing a finite set as a groupoid with only identity morphisms. 

To define groupoid cardinality, suppose $\mathcal{G}$ is a groupoid with countable isomorphism classes such that $\#\text{Aut}(x)$ is finite for all $x\in \mathcal{G}$. Then the \emph{groupoid cardinality} of $\mathcal{G}$ is the element $\chi(\mathcal{G})\in [0,\infty]$ given by
\be\label{GCDEF}
\chi(\mathcal{G})=\sum_{[x]\in [\mathcal{G}]} \frac{1}{\# \text{Aut}(x)},
\ee
where $[\mathcal{G}]$ denotes the set of isomorphism classes of $\mathcal{G}$ and $[x]$ denotes the isomorphism class of the object $x\in \mathcal{G}$. Groupoid cardinality is additive on disjoint unions, multiplicative on products and satisfies other properties one would expect from a cardinality measure. The definition of groupoid cardinality seems to have first appeared in the context of Behrend's trace formula for the Frobenius automorphism on algebraic stacks \cite{BEHRENDTRACE}, and in the case that $\mathcal{G}$ has finitely many objects, $\chi(\mathcal{G})$ coincides with Leinster's definition of Euler characteristic for finite catgeories \cite{LEINSTERECC}. A generalization of groupoid cardinality to Lie groupoids was constructed by Weinstein in \cite{WEINSTEINSTACK}, which may be viewed as the volume of the differentiable stack associated with a Lie groupoid. An equation similar to \eqref{GCDEF} also appears in the definition of the Minkowski-Siegel mass formula for the weights of quadratic forms in a given genus.

In Section~\ref{GCC17}, we prove a characterization theorem for groupoid cardinality in terms of its invariance with respect to categorical equivalence, its additivity over disjoint unions, a continuity property and its behavior with respect to $k$-sheeted coverings. After doing so, the direction of the paper is motivated by the following quote from Baez, Hoffnung and Walker \cite{BAEZGFC}:
\\

``Getting a useful generalization of groupoids for which the cardinality is naturally complex, without putting in the complex numbers `by hand', remains an elusive goal.''
\\

With such a goal in mind, we turn to stuff types in Section~\ref{SEC3}, which were also introduced by Baez and Dolan in \cite{BAEZDOLAN}. Stuff types are groupoids graded by the groupoid $\mathfrak{Fin}$ of bijections between finite sets, and as such, are essentially generalized combinatorial species with the arrows reversed \cite{JOYAL}. Given a stuff type $\mathcal{G}\to \mathfrak{Fin}$, one may associate with $\mathcal{G}$ a formal power series
\[
\mathcal{G}(z)=a_0+a_1z+a_2z^2+\cdots \in \R[[z]]
\]
where $a_n$ is the groupoid cardinality of the full inverse image of an $n$-element set with respect to the functor $\mathcal{G}\to \mathfrak{Fin}$. If it turns out that $\mathcal{G}(z)$ is analytic in a neighborhood of $z=0$ and admits a unique analytic continuation to $z=1$, then we define the \emph{analytic groupoid cardinality} of $\mathcal{G}$ to be the complex number $\chi_a(\mathcal{G})$ obtained via the analytic continuation of $\mathcal{G}(z)$ to $z=1$. Analytic groupoid cardinality then yields a complex value for stuff types whose usual groupoid cardinality diverges. While a construction of `analytically continued cardinalities' for stuff types was first suggested by Baez and Dolan in  \cite{BAEZDOLAN}, to the best of our knowledge a precise statement has yet to be formulated in the literature. In Section~\ref{ATG} we show analytic groupoid cardinality is a measure which restricts to groupoid cardinality for all stuff types which have finite groupoid cardinality, and that it satisfies other properties one expects of a cardinality measure, such as being additive over disjoint unions and multiplicative over products. 

The way in which analytic groupoid cardinality makes sense of a divergent cardinality via analytic continuation is similar in spirit to Berger and Leinster's definition of \emph{series Euler characteristic} for finite categories \cite{LEINSTERECDIV}. For this, the nerve of a category is used to associate with a finite category $A$ a simplicial set $S_A$, with which one may associate the formal power series
\[
A(z)=c_0+c_1z+c_2z^2+\cdots \in \R[[z]]
\]
where $c_i$ is the number of $i$-simplices in $S_A$. The series Euler characteristic $\chi(A)$ is then taken to be the analytic continuation of $A(z)$ to $z=-1$ (which turns out to be well-defined for all $A$). But while such a construction invokes analytic continuation, Berger and Leinster show that the analytic contiuation of $A(z)$ to $z=-1$ may be given by a rational function, thus complex numbers never enter the picture in this context. 

The purely formal nature of the definition of analytic groupoid cardinality along with its departure from the real domain yields little insight as to what analytic groupoid cardinality actually means, which is a topic we address in Section~\ref{AGCNE}. In particular, we define a notion of structural recursion for stuff types which we refer to as \emph{nested equivalence}, and show that if a stuff type admits a nested equivalence, then its analytic groupoid cardinality is necessarily a fixed point of a polynomial associated with the equivalence which we refer to as the \emph{structural polynomial}. As such, the construction of analytic groupoid cardinality is reminiscent of viewing objects of categories which satisfy an isomorphism of the form $P(X)\cong X$ with $P$ a polynomial as categorified complex numbers. Such perspectives go back to the papers \cite{RGates} of R. Gates, \cite{BLASS} of Blass and \cite{LEINSTEROCC} of Fiore and Leinster, all of which -- as well as the present work -- were influenced by the following quote of Lawvere \cite{La91}:
\\

``I was surprised to note that an isomorphism $x=1+x^2$ (leading to complex numbers as Euler characteristics if they don't collapse) always induces an isomorphism $x^7=x$.''

\section{A characterization of groupoid cardinality}\label{GCC17}
Let $\mathcal{G}$ be a groupoid, and let $[\mathcal{G}]$ denote the set of isomorphism classes of $\mathcal{G}$. Unless stated otherwise, we assume $[\mathcal{G}]$ is countable and that $\text{Aut}(x)$ is finite for every object $x$ in $\mathcal{G}$.  

\bd
The \define{groupoid cardinality} $\chi(\mathcal{G})$ of $\mathcal{G}$ is given by
\[
\chi(\mathcal{G})=\sum_{[x]\in [\mathcal{G}]} \frac{1}{\# \text{Aut}(x)} \hspace{0.5mm} .
\]
The groupoid $\mathcal{G}$ will be referred to as \define{tame} if and only if $\chi(\mathcal{G})<\infty$. 
\ed

\bx
Let $\mathfrak{Fin}$ be the groupoid of bijections between finite sets. Then $\mathfrak{Fin}$ is tame, since
\[
\chi\left(\mathfrak{Fin}\right)=\sum_{n=0}^{\infty}\frac{1}{n!}=e
\]
\ex

\bx
Given a finite group $G$, $\B G$ is the groupoid with one object whose morphisms consist of the elements of $G$. It then follows that $\chi(\B G)=\frac{1}{\# G}$.
\ex

\bd
Let $\mathcal{G}$ be a tame groupoid, suppose $f:[\mathcal{G}]\to \N$ is a surjection, and let $\mathcal{G}_n=f^{-1}(n)$ and $\mathcal{G}^{(n)}=\coprod_{i=0}^n\mathcal{G}_i$ for all $n\in \N$. Then the sequence
\[
\mathcal{G}^{(0)}\hookrightarrow \mathcal{G}^{(1)}\hookrightarrow \cdots \hookrightarrow \mathcal{G}^{(n)}\hookrightarrow \cdots
\]
is said to be a \define{filtration} of $\mathcal{G}$.
\ed

\bd\label{PLPD}
Let $F:\mathcal{G}\to \mathcal{H}$ be a functor between groupoids. Then $F$ is said to be a $k$-\define{sheeted covering} if and only if $F$ satisfies the following properties.
\begin{enumerate}[i.]
\item
$F$ is surjective on objects.
\item\label{plp}
Given a morphism $h_1\to h_2$ in $\mathcal{H}$ and an object $g_1$ in $\mathcal{G}$ such that $F(g_1)=h_1$, there exists a unique morphism $g_1\to g_2$ in $\mathcal{G}$ such that $F(g_1\to g_2)=h_1\to h_2$.
\item
The preimage of every object in $\mathcal{H}$ consists of $k$ objects in $\mathcal{G}$.
\end{enumerate} 
\ed

\bn\label{T1}
Groupoid cardinality satisfies the following properties.
\begin{enumerate}[i.]
\item\label{GC1}
If $\mathcal{G}$ is equivalent to  $\mathcal{H}$, then $\chi(\mathcal{G})=\chi(\mathcal{H})$.
\item\label{GC2}
$\chi(\mathcal{G}\sqcup \mathcal{H})=\chi(\mathcal{G})+\chi(\mathcal{H})$ \hspace{2mm} for all tame groupoids $\mathcal{G}$ and $\mathcal{H}$.
\item\label{GC3}
$\chi(\mathcal{G}\times \mathcal{H})=\chi(\mathcal{G})\cdot \chi(\mathcal{H})$ \hspace{3.5mm} for all tame groupoids $\mathcal{G}$ and $\mathcal{H}$.
\item\label{GC4}
If $\mathcal{G}^{(n)}$ is a filtration of $\mathcal{G}$, then $\displaystyle \lim_{n\to \infty}\chi(\mathcal{G}^{(n)})=\chi(\mathcal{G})$.
\item\label{GC5}
If $F:\mathcal{G}\to \mathcal{H}$ is a $k$-sheeted covering between tame groupoids, then $\chi(\mathcal{G})=k\chi(\mathcal{H})$.
\item\label{GC6}
If  $\bullet$ is a groupoid with one object and one morphism, then $\chi(\bullet)=1$.
\item\label{GC7}
If $X$ is a finite set and $G$ is a finite group acting on $X$, then $\chi(X/\!\!/ G)=\frac{\#(X)}{\#(G)}$, where $X/\!\!/ G$ is the associated action groupoid.
\end{enumerate}
\en

\bprf
{\color{white}{you found me!}}
\begin{enumerate}[i.]
\item
Since groupoid cardinality of a groupoid $\mathcal{G}$ is defined in terms of a representative of each isomorphism class of $\mathcal{G}$, the statement follows.
\item
This is proved in Theorem~4 of \cite{MORTON}.
\item
This is proved in Theorem~4 of \cite{MORTON}.
\item
Let $\mathcal{G}$ be a groupoid, and suppose $\mathcal{G}^{(n)}=\coprod_{i=0}^n\mathcal{G}_i$ is a filtration of $\mathcal{G}$. Then
\[
\chi(\mathcal{G})=\sum_{[x]\in [\mathcal{G}]} \frac{1}{\# \text{Aut}(x)}=\sum_{n=0}^{\infty}\chi(\mathcal{G}_n)=\lim_{n\to \infty}\left(\sum_{i=0}^n\chi(\mathcal{G}_i)\right)=\lim_{n\to \infty}\chi(\mathcal{G}^{(n)}),
\]
as desired.
\item
Suppose $F:\mathcal{G}\to \mathcal{H}$ is a $k$-sheeted covering, let $h$ be an object in $\mathcal{H}$, and suppose the $k$ objects in $\mathcal{G}$ which map to $h$ under $F$ belong to $m$ different isomorphism classes, so that we can index the objects which map to $h$ as follows:
\[
g_{11},g_{12},...,g_{1k_1}\in [g_{11}], \quad g_{21},g_{22},...,g_{2k_2}\in [g_{21}], \quad \cdots \quad g_{m1},g_{m2},...,g_{mk_m}\in [g_{m1}]
\]
Now given $i\in \{1,...,m\}$, the path lifting property \ref{plp} of Definition~\ref{PLPD} implies that as a set $\text{Aut}(h)$ is in bijective correspondence with the union
\[
\text{Aut}(g_{i1})\cup\text{Hom}(g_{i1},g_{i2})\cup \cdots \cup \text{Hom}(g_{i1},g_{ik_i})
\]
It then follows that $\#\text{Aut}(g_{i1})\cdot k_i=\#\text{Aut}(h)$ for all $i\in \{1,...,m\}$, and since $k_1+\cdots+k_m=k$, we have
\begin{eqnarray*}
\frac{k}{\#\text{Aut}(h)}=\frac{k_1+\cdots+k_m}{\#\text{Aut}(h)}&=&\frac{k_1}{\#\text{Aut}(h)}+\cdots+\frac{k_m}{\#\text{Aut}(h)} \\
&=&\frac{k_1}{k_1\#\text{Aut}(g_{11})}+\cdots+\frac{k_m}{k_m\#\text{Aut}(g_{m1})} \\
&=&\frac{1}{\#\text{Aut}(g_{11})}+\cdots+\frac{1}{\#\text{Aut}(g_{m1})}, \\
&=&\chi\left(F^{-1}(h)\right),
\end{eqnarray*}
thus 
\[
k\chi(\mathcal{H})=\sum_{[h]\in [\mathcal{H}]}\frac{k}{\#\text{Aut}(h)}=\sum_{[h]\in [\mathcal{H}]}\chi\left(F^{-1}(h)\right)=\chi(\mathcal{G}),
\]
as desired.
\item
The statement follows directly from the definition of groupoid cardinality.
\item
This is proved in Theorem~6 of \cite{MORTON}.
\qedhere
\end{enumerate}
\eprf

\bd
The \define{category of tame groupoids} is the category $\mathfrak{Gpd}_{\mathbb{T}}$ consisting of functors between tame groupoids.
\ed

We now prove a characterization theorem for groupoid cardinality. While a similar characterization for essentially finite groupoids seems to be well-known among experts (though absent from the literature), our characterization holds on the full category of tame groupoids.  

\begin{theorem}\label{ugc}
Suppose $\mu:\emph{Ob}(\mathfrak{Gpd}_{\mathbb{T}})\to \R$ is a function satisfying the following conditions.
\begin{enumerate}[i.]
\item\label{G1}
If $\mathcal{G}$ is equivalent to  $\mathcal{H}$, then $\mu(\mathcal{G})=\mu(\mathcal{H})$.
\item\label{G2}
If  $\bullet$ denotes a groupoid with one object and one morphism, then $\mu(\bullet)=1$.
\item\label{G3}
If $F:\mathcal{G}\to \mathcal{H}$ is a $k$-sheeted covering, then $\mu(\mathcal{G})=k\mu(\mathcal{H})$.
\item\label{G4}
$\mu(\mathcal{G}\sqcup \mathcal{H})=\mu(\mathcal{G})+\mu(\mathcal{H})$ \hspace{2mm} for all tame groupoids $\mathcal{G}$ and $\mathcal{H}$.
\item\label{G5}
If $\mathcal{G}^{(n)}$ is a filtration of $\mathcal{G}$, then $\displaystyle \lim_{n\to \infty}\mu(\mathcal{G}^{(n)})=\mu(\mathcal{G})$.
\end{enumerate}
Then $\mu=\chi$, the groupoid cardinality.
\end{theorem}

We first prove the following statement.

\blem\label{LBG77}
Let $G$ be a finite group, and suppose $\mu:\emph{Ob}(\mathfrak{Gpd}_{\mathbb{T}})\to \R$ is a function satisfying items \ref{G1}, \ref{G2} and \ref{G3} of Theorem~\ref{ugc}. Then $\mu(\B G)=\frac{1}{\# G}$.
\elem

\bprf
Associated with $\B G$ is the $\# G$-sheeted cover $\bE G\to \B G$, where $\bE G$ is the action groupoid associated with the action of $G$ on itself. In particular, the objects of $\bE G$ are the elements of $G$, and given two objects $g,h\in G$, there exists a single morphism $g\to h\in \text{Hom}(g,h)$, which maps to the morphism $hg^{-1}$ under the functor $\bE G\to \B G$. It then follows that all objects in $\bE G$ are isomorphic, and moreover, all the automorphism groups in $\bE G$ are trivial. As such, $\bE G$ is equivalent to a groupoid with one object and one morphism, thus $\mu(\bE G)=1$. Moreover, given a morphism $g\in \B G$ and an object $h$ in $\bE G$, the morphism $h\to gh$ is the unique morphism mapping to $g$ by the functor $\bE G\to \B G$, thus $\bE G\to \B G$ is a $\# G$-sheeted covering map. It then follows that $\mu(\bE G)=\# G\mu(\B G)$, and since $\mu(\bE G)=1$, we have $\mu(\B G)=\frac{1}{\# G}$, as desired.
\eprf

\bprf[Proof of Theorem~\ref{ugc}]
Let $\mathcal{G}$ be a tame groupoid. Since every groupoid is equivalent to its skeleton, by the tameness of $\mathcal{G}$ we may assume its skeleton $\mathbb{G}$ is a countable disjoint union $\mathbb{G}=\coprod_{i\in I}\star_i$, with $\star_i$ a groupoid with one object $x_i$ for all $i\in I$. Item \ref{GC1} of Proposition~\ref{T1} then yields
\begin{equation}\label{e1}
\chi(\mathcal{G})=\chi(\mathbb{G})=\sum_{i\in I} \frac{1}{\#\text{Aut}(x_i)}
\end{equation} 
If the index set $I$ is in fact finite, we then have
\begin{equation}\label{e2}
\mu(\mathcal{G})\overset{(\text{item} \hspace{1mm} \ref{G1})}=\mu(\mathbb{G})\overset{(\text{item} \hspace{1mm} \ref{G4})}=\sum_{i\in I}\chi(\star_i)\overset{(\text{Lemma}~\ref{LBG77})}=\sum_{i\in I} \frac{1}{\# \text{Aut}(x_i)}\overset{\eqref{e1}}=\chi(\mathcal{G}),
\end{equation}
as desired. 

Now suppose the index set $I$ is countably infinite, and choose a bijection between $I$ and $\N$.  It then follows that $\mathbb{G}=\coprod_{n=0}^{\infty}\star_n$, so that $\mathbb{G}^{(n)}=\coprod_{j=0}^n\star_j$ is a filtration of $\mathbb{G}$. We then have
\[
\mu(\mathbb{G}^{(n)})\overset{(\text{item}~\ref{G4})}=\sum_{j=0}^n\mu(\star_j)\overset{(\text{Lemma}~\ref{LBG77})}=\sum_{j=0}^n\frac{1}{\# \text{Aut}(x_j)}=\sum_{[x]\in [\mathbb{G}^{(n)}]} \frac{1}{\#\text{Aut}(x)}=\chi(\mathbb{G}^{(n)}),
\]
thus
\[
\mu(\mathcal{G})\overset{(\text{item} \hspace{1mm} \ref{G1})}=\mu(\mathbb{G})\overset{(\text{item} \hspace{1mm} \ref{G5})}=\lim_{n\to \infty}\mu(\mathbb{G}^{(n)})=\lim_{n\to \infty}\chi(\mathbb{G}^{(n)})=\chi(\mathbb{G})=\chi(\mathcal{G}),
\]
as desired.
\eprf

\section{Generating series of stuff types}\label{SEC3}

In the same paper in which they introduced groupoid cardinality \cite{BAEZDOLAN}, Baez and Dolan introduced \emph{stuff types} with a view towards categorifying the Fock space associated with the quantum harmonic oscillator. Stuff types are essentially generalized combinatorial species, and for a certain class of stuff types -- which we refer to as \emph{relatively tame} -- one may associate with each stuff type a formal power series which may be viewed as a decategorification of the stuff type. In Section~\ref{ATG}, we use power series associated with stuff types to define \emph{analytic groupoid cardinality}, which is a complex-valued extension of groupoid cardinality to a certain class of relatively tame groupoids.

\bd
Let $\mathfrak{Fin}$ be the groupoid of bijections between finite sets. A \define{stuff type} is a groupoid $\mathcal{G}$ endowed with a functor $\mathcal{G}\to \mathfrak{Fin}$. The $n$th \define{graded component} of a stuff type $\mathcal{G}$ is the groupoid $\mathcal{G}_n$ whose objects consist of the objects of $\mathcal{G}$ which map to $n$-element sets under the functor $\mathcal{G}\to \mathfrak{Fin}$, together with all the associated hom-sets in $\mathcal{G}$. A stuff type $\mathcal{G}$ then admits the structure of the coproduct 
\[
\mathcal{G}=\coprod_{n=0}^{\infty}\mathcal{G}_n
\]
and $\mathcal{G}^{(n)}=\coprod_{i=0}^{n}\mathcal{G}_i$ is said to be the \define{canonical filtration} of $\mathcal{G}$.
\ed

\bd
Let $\mathcal{G}$ be a stuff type. If $\mathcal{G}_d$ is non-empty for some $d\in \N$ while $\mathcal{G}_n$ is the empty groupoid for $n\neq d$, then $\mathcal{G}$ is said to be of \define{degree} $d$. As such, if $\mathcal{G}_n$ is non-empty then $\mathcal{G}_n$ is necessarily a stuff type of degree $n$. An object of $\mathcal{G}_n$ will be referred to as an \define{object of degree $n$}. Objects of degree 0 and 1 will be referred to as \define{null objects} and \define{pointed objects}, respectively.
\ed

\bx
Let $k\in \N$, and let $Z^k\to \mathfrak{Fin}$ be the stuff type which takes a totally ordered, $k$-element set to its underlying set (the notation $Z$ here is not to confused with the standard notation for the integers, namely $\mathbb{Z}$). Then $Z^k$ is of degree $k$ for all $k\in \N$.
\ex

\bx
Let $\mathcal{G}$ be the groupoid of labeled simple graphs, with morphisms corresponding to relabeling of vertices. Then $\mathcal{G}$ admits the structure of a stuff type by sending a labeled simple graph to its set of vertices, and sending a morphism to the corresponding bijection between the vertices. In such a case, $\mathcal{G}_n$ is the groupoid of labeled simple graphs on $n$ vertices, while $[\mathcal{G}_n]$ may be identified with the set of \emph{unlabeled} simple graphs on $n$ vertices.
\ex

\bx
Let $\mathcal{G}$ be a tame groupoid, and let $\bold{X}_{\mathcal{G}}:\mathcal{G}\to \mathfrak{Fin}$ be the functor defined as follows.
On objects, $\bold{X}_{\mathcal{G}}(x)=\underline{\text{Aut}(x)}$ for all objects $x$ in $\mathcal{G}$, where $\underline{\text{Aut}(x)}$ denotes the underlying set of $\text{Aut}(x)$ (which is finite since $\mathcal{G}$ is tame). Given objects $x$ and $y$ of $\mathcal{G}$ and a morphism $f\in \text{Hom}(x,y)$, the morphism $\bold{X}_{\mathcal{G}}(f):\underline{\text{Aut}(x)}\to \underline{\text{Aut}(y)}$ is the bijection given by
\[
g\mapsto f\circ g \circ f^{-1} \in \underline{\text{Aut}(y)}
\]
It follows directly from the definitions that $\bold{X}_{\mathcal{G}}$ is indeed a functor, so that every tame groupoid admits a stuff type structure. 
\ex

We now show how stuff types may be viewed as a categorification of formal power series.

\bd
The \define{category of stuff types} is the category $\mathfrak{Fin}[Z]$ whose objects are stuff types, and given stuff types $\bf{X}:\mathcal{G}\to \mathfrak{Fin}$ and $\bf{Y}:\mathcal{H}\to \mathfrak{Fin}$, a morphism from $\bf{X}$ to $\bf{Y}$ consists of a functor $T:\mathcal{G}\to \mathcal{H}$ such that $\bf{X}$ is naturally isomorphic $\bold{Y}\circ T$. In particular, a morphism between $\bf{X}$ and $\bf{Y}$ requires the datum of a functor $T:\mathcal{G}\to \mathcal{H}$ and morphisms $\tau_x:\bold{X}(x)\to \bold{Y}(T(x))$ for all $x\in \text{Ob}(\mathcal{G})$, such that for every morphism $f:x\to y \in \text{Mor}(\mathcal{G})$ we have the following commutative diagram in $\mathfrak{Fin}$.
\begin{equation}\label{ntd}
\xymatrix{
\bold{X}(x) \ar[d]_{\bold{X}(f)} \ar[r]^{\tau_x} & \bold{Y}(T(x)) \ar[d]^{\bold{Y}(T(f))} \\
\bold{X}(y) \ar[r]_{\tau_y} & \bold{Y}(T(y)) \\
}
\end{equation}
Two stuff types $\bf{X}:\mathcal{G}\to \mathfrak{Fin}$ and $\bf{Y}:\mathcal{H}\to \mathfrak{Fin}$ are said to be \define{equivalent} if and only if there exists a morphism $\bold{X}\to \bold{Y}$ such that the underlying functor $T:\mathcal{G}\to \mathcal{H}$ is an equivalence of categories. 
\ed

\begin{proposition}\label{p77}
If the stuff types $\bold{X}:\mathcal{G}\to \mathfrak{Fin}$ and $\bold{Y}:\mathcal{H}\to \mathfrak{Fin}$ are equivalent, then $\mathcal{G}_n$ is equivalent to $\mathcal{H}_n$ for all $n\geq 0$.
\end{proposition}
\begin{proof}
Suppose $(T,\tau):\bold{X}\to \bold{Y}$ is a morphism between stuff types $\bold{X}:\mathcal{G}\to \mathfrak{Fin}$ and $\bold{Y}:\mathcal{H}\to \mathfrak{Fin}$ such that $T:\mathcal{G}\to \mathcal{H}$ is an equivalence of categories. By diagram \eqref{ntd} $T$ necessarily sends objects of $\mathcal{G}_n$ to objects of $\mathcal{H}_n$, thus $\left.T\right|_{\mathcal{G}_n}:\mathcal{G}_n\to \mathcal{H}_n$ is an equivalence.
\end{proof}

\bd
Let $\bold{X}:\mathcal{G}\to \mathfrak{Fin}$ and $\bold{Y}:\mathcal{H}\to \mathfrak{Fin}$ be stuff types. A morphism $(T,\tau):\bold{X}\to \bold{Y}$ is said to be a $k$-\define{sheeted covering} if and only if the functor $T:\mathcal{G}\to \mathcal{H}$ is a $k$-sheeted covering.
\ed

We now recall the monoidal operations of addition and partitional product in $\mathfrak{Fin}[Z]$. 

\bd
Let $\bold{X}:\mathcal{G}\to \mathfrak{Fin}$ and $\bold{Y}:\mathcal{H}\to \mathfrak{Fin}$ be two objects in $\mathfrak{Fin}[Z]$. 
\begin{enumerate}[i.]
\item
\underline{Addition}: The \define{sum} of $\bold{X}$ and $\bold{Y}$ is the stuff type $\bold{X}+\bold{Y}:\mathcal{G}\coprod \mathcal{H}\to \mathfrak{Fin}$, with $(\bold{X}+\bold{Y})(x)=\bold{X}(x)$ if $x\in \text{Ob}(\mathcal{G})$ and $(\bold{X}+\bold{Y})(y)=\bold{Y}(y)$ if $y\in \text{Ob}(\mathcal{H})$ (and similarly for the morphisms). The map $(\bold{X},\bold{Y})\mapsto \bold{X}+\bold{Y}$ will be referred to as \define{addition}, and the additive identity is the unique functor $\bold{0}$ from the empty groupoid to $\mathfrak{Fin}$. 
\item
\underline{Partitional Product}: The \define{partitional product} (or \define{Cauchy product}) of $\bold{X}$ and $\bold{Y}$ is the stuff type $\bold{X}\cdot \bold{Y}:\mathcal{G}\times \mathcal{H}\to \mathfrak{Fin}$, where $(\bold{X}\cdot \bold{Y})(x,y)=\bold{X}(x)\sqcup \bold{Y}(y)$. As for the morphisms, if $(x\to x',y\to y')$ is a morphism in $\mathcal{G}\times \mathcal{H}$, then $(\bold{X}\cdot \bold{Y})(x\to x',y\to y')$ gets sent to the induced bijection $\bold{X}(x)\sqcup \bold{Y}(y)\to \bold{X}(x')\sqcup \bold{Y}(y')$. The multiplicative identity with respect to the partitional product is unique up to natural isomorphism, which is necessarily of the form $\mathbbm{1}:\circ \to \mathfrak{Fin}$, where $\circ$ is a groupoid with a single null object $\circ$ and single morphism, $\mathbbm{1}(\circ)=\varnothing$ and $\mathbbm{1}(\circ\to \circ)=\varnothing\to \varnothing$.
\end{enumerate}
\ed

\br
Addition and partitional products of stuff types are commutative up to natural isomorphism, and while the addition of stuff types is the categorical coproduct in $\mathfrak{Fin}[Z]$, the partitional product is \emph{not} the categorical product in $\mathfrak{Fin}[Z]$. If $\bold{X}:\mathcal{G}\to \mathfrak{Fin}$ and $\bold{Y}:\mathcal{H}\to \mathfrak{Fin}$ are objects in $\mathfrak{Fin}[Z]$, we will often denote $\bold{X}+\bold{Y}$ and $\bold{X}\cdot \bold{Y}$ by $\mathcal{G}+\mathcal{H}$ and $\mathcal{G}\cdot \mathcal{H}$ respectively, with the overlying functors being implicit. 
\er

\begin{proposition}\label{p19}
Let $\mathcal{G}, \mathcal{H} \in \emph{Ob}(\mathfrak{Fin}[Z])$. Then for all $n\geq 0$ we have
\begin{enumerate}[i.]
\item\label{COEFF1}
$(\mathcal{G}+\mathcal{H})_n=\mathcal{G}_n\coprod\mathcal{H}_n$
\item\label{COEFF3}
$(\mathcal{G}\cdot \mathcal{H})_n=\coprod_{i+j=n}\mathcal{G}_i\times \mathcal{H}_j$
\end{enumerate}
\end{proposition}
\begin{proof}
The proposition follows directly from the definitions of addition and partitional product. 
\end{proof}

\bd
A stuff type $\mathcal{G}\to \mathfrak{Fin}$ will be referred to as \define{relatively tame} if and only if $\mathcal{G}_n$ is tame for all $n\in \N$.
\ed

\br
If $\mathcal{G}$ is tame, then $\bold{X}:\mathcal{G}\to \mathfrak{Fin}$ is relatively tame for all $\bold{X}$.
\er

\bd
The \define{generating series} of a relatively tame stuff type $\mathcal{G}\to \mathfrak{Fin}$ is the formal power series $\mathcal{G}(z)\in \R[[z]]$ given by
\[
\mathcal{G}(z)=\sum_{n=0}^{\infty}\chi(\mathcal{G}_n)z^n
\] 
\ed

\bx\label{E971}
Suppose $\bold{X}:\mathcal{G}\to \mathfrak{Fin}$ is a relatively tame stuff type with $\bold{X}$ faithful, let $[n]$ denote an $n$-element set for all $n\geq 0$, and let $\bold{X}^{-1}([n])$ denote the set of objects which map to the set $[n]$ under $\bold{X}$. Then $\mathcal{G}_n$ is equivalent to the action groupoid $\bold{X}^{-1}([n])/\!\!/ \mathcal{S}_n$ (see \cite{MORTON}, Remark~5), thus
\[
\mathcal{G}(z)=\sum_{n=1}^{\infty}\frac{\#\bold{X}^{-1}([n])}{n!}z^n
\] 
As such, $\mathcal{G}(z)$ coincides with the exponential generating function for $\mathcal{G}$-structures on finite sets.  
\ex

\bn\label{p17}
The generating series of a relatively tame stuff type satisfies the following properties.
\begin{enumerate}[i.]
\item\label{gs0}
If $\mathcal{G}$ is equivalent to  $\mathcal{H}$, then $\mathcal{G}(z)=\mathcal{H}(z)$.
\item\label{gs2}
If $T:\mathcal{G}\to \mathcal{H}$ is a $k$-sheeted covering between relatively tame stuff types, then $\mathcal{G}(z)=k\mathcal{H}(z)$.
\item\label{gs1}
$(\mathcal{G}+ \mathcal{H})(z)=\mathcal{G}(z)+\mathcal{H}(z)$ \hspace{2mm} for all relatively tame stuff types $\mathcal{G}$ and $\mathcal{H}$.
\item\label{gs3}
$(\mathcal{G}\cdot \mathcal{H})(z)=\mathcal{G}(z)\cdot \mathcal{H}(z)$ \hspace{2mm} for all relatively tame stuff types $\mathcal{G}$ and $\mathcal{H}$, where $\mathcal{G}(z)\cdot \mathcal{H}(z)$ denotes the usual Cauchy product of power series.
\item\label{gs4}
If  $\mathcal{G}_n$ is a stuff type of degree $n$ with one object and one morphism, then $\mathcal{G}_n(z)=z^n$.
\end{enumerate}
\en

\bprf
If $\mathcal{G}$ and $\mathcal{H}$ are equivalent, then by Proposition~\ref{p77} we have $\mathcal{G}_n$ is equivalent to $\mathcal{H}_n$ for all $n\geq 0$. It then follows by item \ref{GC1} of Proposition~\ref{T1} that $\chi(\mathcal{G}_n)=\chi(\mathcal{H}_n)$ for all $n\geq 0$, thus $\mathcal{G}_{\mu}(z)=\mathcal{H}_{\mu}(z)$. Since $T:\mathcal{G}\to \mathcal{H}$ is a $k$-sheeted covering we have $\mathcal{G}_n=k\mathcal{H}_n$ for all $n$, from which item \ref{gs2} follows. Items \ref{gs1} - \ref{gs3} follow from items \ref{COEFF1} and \ref{COEFF3} of Proposition~\ref{p19}. Item \ref{gs4} follows directly from the definition of generating series of a stuff type.
\eprf

\section{Analytically tame stuff types}\label{ATG}

We now define \emph{analytic groupoid cardinality}, which is a complex-valued extension of groupoid cardinality to a class of relatively tame stuff types which we refer to as \emph{analytically tame}. 

\bd
A relatively tame stuff type $\mathcal{G}$ will be referred to as \define{analytically tame} if and only if its generating series $\mathcal{G}(z)$, viewed as a function of the complex variable $z$, is analytic in a neighborhood of $z=0$ and admits a unique analytic continuation to  $z=1$. In such a case, denote the value at $z=1$ of the analytic continuation of $\mathcal{G}(z)$ to  $z=1$ by $\mathcal{G}_a(1)$. We then define the \define{analytic groupoid cardinality} $\chi_a(\mathcal{G})$ of an analytically tame stuff type $\mathcal{G}$ to be the complex number given by $\chi_a(\mathcal{G})=\mathcal{G}_a(1)$.
\ed

\bx\label{BRT17}
Let $\mathcal{B}$ be the stuff type of labeled, planar, binary rooted trees with morphisms corresponding to relabeling of vertices, and let $F:\mathcal{B}\to \mathfrak{Fin}$ be the functor which takes a binary rooted tree to its set of vertices and takes a morphism to the underlying bijection between the vertices. Then $F$ is faithful, so by Example~\ref{E971} we have
\[
\mathcal{B}(z)=\sum_{n=0}^{\infty}c_nz^n,
\]
where $c_n=\frac{1}{n+1}{2n\choose n}$ is the $n$th Catalan number. The power series $\mathcal{B}(z)$ is analytic in a neighborhood of $z=0$, and admits a unique analytic continuation to $z=1$ given by
\[
\mathcal{B}(z)=\frac{1-\sqrt{1-4z}}{2z}
\]
It then follows that $\mathcal{B}$ is analytically tame, and $\chi_a(\mathcal{B})=\frac{1}{2}-\frac{\sqrt{3}}{2}i$.
\ex

\bx\label{BSZ19}
Let $\mathcal{G}$ be the stuff type corresponding to the structure of being a binary string with no consecutive zeros. An object of $\mathcal{G}$ is a triple $(S,\mathscr{O}_S,\tau)$, where $S$ is a finite set, $\mathscr{O}_S:S\to \{1,...,n\}$ is an ordering of $S$ and $\tau:\{1,...,n\}\to \{0,1\}$ is a binary string with no consecutive zeros. A morphism in this category only exists between objects of the form $(S,\mathscr{O}_S,\tau)$ and $(T,\mathscr{O}_T,\tau)$ (so $\tau$ must be the same for both objects), and consists of a bijection $f:S\to T$ such that $\tau\circ \mathscr{O}_S=\tau\circ \mathscr{O}_T\circ f$, as in the following commutative diagram.
\[
\xymatrix{
S \ar[rr]^{f} \ar[dr]_{\mathscr{O}_S} & & T \ar[dl]^{\mathscr{O}_T} \\
& \{1,...,n\} \ar[d]_{\tau}&  \\
&\{0,1\} & \\
}
\]
Since a binary string with no consecutive zeros is either the empty string, $0$, or of the form $1\star b$ or $01\star b$ for some binary string $b$ with no repeated zeros ($\star$ denotes concatenation), it follows that 
\[
\mathcal{G}\cong \mathbbm{1}+Z+(Z+Z^2)\mathcal{G},
\]
where we recall $Z^k$ is the degree $k$ stuff type of being a totally ordered, $k$-element set for all $k\in \N$. We then have
\[
\mathcal{G}(z)=1+z+(z+z^2)\mathcal{G}(z),
\]
thus
\[
\mathcal{G}(z)=\frac{1+z}{1-(z+z^2)},
\]
which yields $\chi_a(\mathcal{G})=-2$.
\ex

\bd
The \define{category of analytically tame stuff types} is the category $\mathfrak{Fin}_a[Z]$ consisting of stuff type morphisms between analytically tame stuff types.
\ed

\bn\label{T171}
Analytic groupoid cardinality satisfies the following properties.
\begin{enumerate}[i.]
\item\label{AGC1}
If an analytically tame stuff type $\mathcal{G}$ is equivalent to $\mathcal{H}$, then $\mathcal{H}$ is analytically tame and $\chi_a(\mathcal{G})=\chi_a(\mathcal{H})$.
\item
If an analytically tame stuff type $\mathcal{G}$ is in fact tame, then $\chi_a(\mathcal{G})=\chi(\mathcal{G})$.
\item
If $\mathcal{G}\to \mathcal{H}$ is a $k$-sheeted covering map between analytically tame stuff types, then $\chi_a(\mathcal{G})=k\chi_a(\mathcal{H})$.
\item\label{AGC2}
$\chi_a(\mathcal{G}+ \mathcal{H})=\chi_a(\mathcal{G})+\chi_a(\mathcal{H})$ \hspace{2mm} for all analytically tame stuff types $\mathcal{G}$ and $\mathcal{H}$.
\item\label{AGC3}
$\chi_a(\mathcal{G}\cdot \mathcal{H})=\chi_a(\mathcal{G})\cdot \chi_a(\mathcal{H})$ \hspace{3.5mm} for all analytically tame stuff types $\mathcal{G}$ and $\mathcal{H}$.
\end{enumerate}
\en

\bprf
\begin{enumerate}[i.]
\item
By item \ref{gs0} of Proposition~\ref{p17} $\mathcal{G}(z)=\mathcal{H}(z)$, and since $\mathcal{G}$ is analytically tame it follows that $\mathcal{H}$ is analytically tame as well. Moreover, since $\mathcal{G}(z)=\mathcal{H}(z)$ we have $\chi_a(\mathcal{G})=\chi_a(\mathcal{H})$.
\item
Let $\mathcal{G}^{(n)}=\coprod_{i=0}^{n}\mathcal{G}_i$ be the canonical filtration of $\mathcal{G}$ associated with the functor $\mathcal{G}\to \mathfrak{Fin}$ which endows $\mathcal{G}$ with the structure of a stuff type. Since $\mathcal{G}$ is tame, we have
\[
\chi(\mathcal{G})=\lim_{n\to \infty}\chi(\mathcal{G}^{(n)})=\lim_{n\to \infty}\sum_{i=0}^{n}\chi(\mathcal{G}_n)=\mathcal{G}(1),
\]
where the first and second equalities follow from items \ref{G4} and \ref{G2} of Proposition~\ref{T1} respectively. And since $\mathcal{G}(1)$ is necessarily the unique analytic continuation of $\mathcal{G}(z)$ to $z=1$, it follows that $\chi(\mathcal{G})=\chi_a(\mathcal{G})$.
\item
The statement follows directly from item \ref{gs2} of Proposition~\ref{p17}.
\item
Since $\mathcal{G}(z)$ and $\mathcal{H}(z)$ admit unique analytic continuations to $z=1$, it follows that $\mathcal{G}(z)$ and $\mathcal{H}(z)$ may be analytically continued to $z=1$ along a common domain, thus by item \ref{gs1} of Proposition~\ref{p17} $(\mathcal{G}+ \mathcal{H})(z)$ admits an analytic continuation to $z=1$ given by $\mathcal{G}_a(1)+\mathcal{H}_a(1)$. Moreover, given an analytic continuation of $(\mathcal{G}+ \mathcal{H})(z)$ along a path from $z=0$ to $z=1$, one may continue both $\mathcal{G}(z)$ and $\mathcal{H}(z)$ along a homologous path, and as such, the continuation of $(\mathcal{G}+ \mathcal{H})(z)$ along a such a path must necessarily coincide with $\mathcal{G}(z)+\mathcal{H}(z)$ continued along a homologous path, thus $(\mathcal{G}+ \mathcal{H})(z)$ admits a unique analytic continuation to $z=1$ given by $\mathcal{G}_a(1)+\mathcal{H}_a(1)$. We then have $\chi_a(\mathcal{G}+ \mathcal{H})=\chi_a(\mathcal{G})+\chi_a(\mathcal{H})$, as desired.
\item
The proof is similar to that of the proof of item \ref{AGC2}.
\qedhere
\end{enumerate}
\eprf

\section{Analytic groupoid cardinality and nested equivalence}\label{AGCNE}

We now address the issue of the meaning of analytic groupoid cardinality, and what it may tell us about the structure of an analytically tame stuff type. In particular, we introduce the notion of `nested equivalence', which is an equivalence between stuff types which reflects a sort of structural recursion within a stuff type. The prototypical example of nested equivalence is that of the stuff type of binary trees, as a binary tree may be decomposed recursively into two binary trees joined together at a common vertex. The existence of a nested equivalence induces a set bijection between the associated isomorphism classes, whose codomain is necessarily a polynomial in the domain (in the context of type theory such bijections were referred to as `generic recursive polynomial types' in \cite{Fiore1}). We then refer to such a polynomial as the `structural polynomial' associated with the nested equivalence, and then show that in such a case the analytic groupoid cardinality of a stuff type which admits such a nested equivalence is a fixed point of the structural polynomial associated with the equivalence. As such, for analytically tame stuff types which admit nested equivalence, one may view analytic groupoid cardinality as a numerical avatar of structural recursion within the isomorphism classes of the stuff type.

\bd
Let $p(z)=a_0+a_1z+\cdots+a_kz^k\in \N[z]$. 
\begin{enumerate}[i.]
\item
The stuff type $p(Z)$ is given by $p(Z)=a_0\mathbbm{1}+a_1Z+\cdots+a_kZ^k$, where we recall $Z^k$ is the stuff-type of being a totally ordered, $k$-element set.
\item
If $S$ is a set, then $p(S)$ is set given by $p(S)=a_0\{\bullet\}+a_1S+\cdots+a_kS^k$, where $+$ denotes disjoint union. 
\end{enumerate}
\ed

In Example~\ref{BSZ19} we saw that the analytically tame stuff type $\mathcal{G}$ of being a binary string with no consecutive zeros admits an equivalence of the form  
\[
\mathcal{G}\cong \mathbbm{1}+Z+(Z+Z^2)\mathcal{G},
\]
which motivates the following definition.

\bd\label{NXEz17}
Let $\mathcal{G}$ be a stuff type. A \define{nested equivalence} is an equivalence of the form
\be\label{NXSp999}
\mathcal{G}\cong p_0(Z)+p_1(Z)\mathcal{G}+\cdots +p_m(Z)\mathcal{G}^m,
\ee
where $p_0(z),...,p_m(z)\in \Z_2[z]$, with $p_m$ not equal to the zero polynomial.
\ed

\br
Replacing the coefficients $p_0(Z),...,p_m(Z)\in \mathfrak{Fin}_a[Z]$ in \eqref{NXSp999} of Definition~\ref{NXEz17} with arbitrary stuff types $A_0,...,A_n \in \mathfrak{Fin}_a[Z]$ such that $\chi_a(A_i)\in \N$ would work just as fine for our purposes. However, we work with polynomials in $Z$ for added clarity.
\er

If $F:\mathcal{G}\to \mathcal{H}$ is an equivalence between discrete groupoids, then the function $[F]:[\mathcal{G}]\to [\mathcal{H}]$ given by $[F]([x])=[F(x)]$ is immediately seen to be a bijection. In the case of nested equivalence, i.e., when $\mathcal{H}=p_0(Z)+p_1(Z)\mathcal{G}+\cdots +p_m(Z)\mathcal{G}^m$, the associated bijection is of the form $[F]:[\mathcal{G}]\to p_0(1)+p_1(1)[\mathcal{G}]+\cdots +p_m(1)[\mathcal{G}]^m$. Such bijections are referred to as `generic recursive polynomial type' in \cite{Fiore1}, and will play a crucial role moving forward. 


\bd
Let $\mathcal{G}$ be a stuff type, and suppose $F:\mathcal{G}\to p_0(Z)+p_1(Z)\mathcal{G}+\cdots +p_m(Z)\mathcal{G}^m$ is a nested equivalence. The associated bijection $[F]:[\mathcal{G}]\to p_0(1)+p_1(1)[\mathcal{G}]+\cdots +p_m(1)[\mathcal{G}]^m$ given by $[F]([x])=[F(x)]$ will be referred to as the \define{structural bijection} associated with $F$, and the polynomial $P_F(z)=p_0(1)+p_1(1)z+\cdots +p_m(1)z^m\in \Z[z]$ will be referred to as the \define{structural polynomial} associated with $F$.
\ed

\bt
Let $\mathcal{G}$ be an analytically tame stuff type, and suppose $F:\mathcal{G}\to p_0(Z)+p_1(Z)\mathcal{G}+\cdots +p_m(Z)\mathcal{G}^m$ is a nested equivalence. Then $\chi_a(\mathcal{G})$ is a fixed point of the structural polynomial $P_F(z)$, i.e., $\chi_a(\mathcal{G})=P_F(\chi_a(\mathcal{G}))$.
\et

\bprf
Since $F:\mathcal{G}\to p_0(Z)+p_1(Z)\mathcal{G}+\cdots +p_m(Z)\mathcal{G}^m$ is an equivalence, Proposition~\ref{T171} yields
\begin{eqnarray*}
\chi_a(\mathcal{G})&=&\chi_a\left(p_0(Z)+p_1(Z)\mathcal{G}+\cdots +p_m(Z)\mathcal{G}^m\right) \\
&=&p_0(1)+p_1(1)\chi_a(\mathcal{G})+\cdots +p_m(1)\chi_a(\mathcal{G})^m \\
&=&P_F(\chi_a(\mathcal{G})),
\end{eqnarray*}
as desired.
\eprf

\bx
An object of the stuff type $\mathcal{B}$ of labeled, planar, binary rooted trees is either the empty tree, or an ordered pair of trees joined together at its root. It then follows that there exists a nested equivalence of the form 
\[
F:\mathcal{B}\to \mathbbm{1}+Z\mathcal{B}^2,
\]
whose associated structural polynomial is $P_F(z)=1+z^2$. In Example~\ref{BRT17}  we determined that $\chi_a(\mathcal{B})=\frac{1}{2}-\frac{\sqrt{3}}{2}i$, thus
\[
P_F(\chi_a(\mathcal{G}))=1+\left(\frac{1}{2}-\frac{\sqrt{3}}{2}i\right)^2=\frac{1}{2}-\frac{\sqrt{3}}{2}i=\chi_a(\mathcal{G}),
\]
as expected. The structural bijection $[\mathcal{B}]\to 1+[\mathcal{B}]^2$ is then given by
\[
\epsilon \longmapsto \epsilon, \quad T_l-\circ-T_r\longmapsto (T_l,T_r),
\]
where $\epsilon$ denotes the empty tree and $T_l-\circ-T_r$ denotes the non-empty tree with left and right subtrees $T_l$ and $T_r$ respectively.
\ex

\bx
In Example~\ref{BSZ19} we determined that the stuff type $\mathcal{G}$ corresponding to the structure of being a binary string with no consecutive zeros is such that $\chi_a(\mathcal{G})=-2$, which follows from the fact that $\mathcal{G}$ admits a nested equivalence of the form 
\be\label{NE1999}
F:\mathcal{G}\to \mathbbm{1}+Z+(Z+Z^2)\mathcal{G} \hspace{0.5mm} .
\ee
The associated structural polynomial is then $P_F(z)=2+2z$, for which $\chi_a(\mathcal{G})=-2$ is a fixed point. Note that in this case an isomorphism class $[g]\in [\mathcal{G}]$ is what we usually think of as a binary string, as binary strings are unlabeled structures. The structural bijection $[\mathcal{G}]\to 2+2[\mathcal{G}]$ associated with the nested equivalence \eqref{NE1999} is then given by 
\[
\epsilon \longmapsto \epsilon, \quad 0\longmapsto 0 , \quad 1\star b\longmapsto b, \quad 01\star b\longmapsto b,
\]
where $\epsilon$ denotes the empty string and $\star$ denotes concatenation.
\ex

\bx
Let $\mathcal{M}$ denote the groupoid of non-empty, labeled Motzkin trees. A Motzkin tree is a rooted, planar, unary- and/or binary-branching tree, so that each node has out-degree 0, 1 or 2. For binary-branching, one of the output nodes is associated with the left while the other is associated with the right, and for unary-branching the unique output node is not associated with a direction. The generating series for Motzkin trees then coincides with the generating function for the Motzkin numbers (multiplied by $z$), which is given by
\[
\mathcal{M}(z)=\frac{1-z-\sqrt{(1-z)^2-4z^2}}{2z}=z+z^2+2z^3+4z^4+9z^5+21z^6+51z^7+\cdots ,
\]
thus $\chi_a(\mathcal{M})=-i$. As the root of a non-empty Motzkin tree is either non-branching (so that the tree consists of a single node), unary-branching, or binary-branching, it follows that Motzkin trees admit a nested equivalence of the form \[
F:\mathcal{M}\to Z+Z\mathcal{M}+Z\mathcal{M}^2.
\]
The structural bijection $[\mathcal{M}]\to 1+[\mathcal{M}]+[\mathcal{M}]^2$ is then given by
\[
\circ \longmapsto \circ, \quad \circ - T\longmapsto T, \quad T_l-\circ-T_r\longmapsto (T_l,T_r),
\]
where $\circ$ denotes the root of a Motzkin tree, $\circ - T$ denotes a Motzkin tree with a unary branching root, and $T_l-\circ-T_r$ denotes a Motzkin tree with a binary branching root. The structural polynomial is $P_F(z)=1+z+z^2$, for which $\chi_a(\mathcal{M})=-i$ is a fixed point. It is interesting to note that if we add the empty tree to this structure the resulting groupoid cardinality is $1-i$, but a nested equivalence realizing this complex cardinality is not evident (if one even exists). We address this point more generally in Question~\ref{qxZ57}.
\ex

\bx
Let $\mathscr{O}$ denote the stuff type of totally ordered finite sets. As a totally ordered $n$-element set  for $n>0$ is its first element joined together with a totally ordered $n-1$-element set, it follows that there exists a nested equivalence of the form
\[
F:\mathscr{O}\to \mathbbm{1}+Z\mathscr{O},
\]
which yields $\mathscr{O}(z)=\frac{1}{1-z}$. Since $\mathscr{O}(z)$ has a pole at $z=1$, it follows that $\mathscr{O}$ is \emph{not} analytically tame. Moreover, the associated structural polynomial is $p_F(z)=1+z$, which has no fixed points, thus reflecting the fact that $\mathscr{O}$ is not analytically tame. 
\ex

\br
As totally ordered finite sets may be thought of as unary trees, there is nothing more "infinite" about the structure of totally ordered finite sets compared to that of binary trees, which have a non-infinite analytic groupoid cardinality of $\frac{1}{2}-\frac{\sqrt{3}}{2}i$. As such, perhaps it is useful to also consider $\mathbb{P}^1$-valued analytic groupoid cardinality, in order to incorporate simple examples such as totally ordered finite sets into the framework. 
\er

\br
Let $P_F(z)$ be the structural polynomial associated with a nested equivalence $F:\mathcal{G}\to p_0(Z)+p_1(Z)\mathcal{G}+\cdots +p_m(Z)\mathcal{G}^m$. If the degree of $P_F(z)$ is at least 2, then it follows from the work of Fiore and Leinster \cite{LEINSTEROCC} that if
\[
z=P_F(z)\implies q(z)=r(z)\in \N[z]
\]
(with $q(z)$ and $r(z)$ non-constant), then there exists a bijection $q([\mathcal{G}])\to r([\mathcal{G}])$ built up out of copies of the structural bijection $[\mathcal{G}]\to P_F([\mathcal{G}])$. In such a case it then follows from Proposition~\ref{T171} that $q(\chi_a(\mathcal{G}))=r(\chi_a(\mathcal{G}))$. 
\er

The previous remark leads naturally to the following:

\begin{question}\label{qxZ57}
Let $\mathcal{G}$ be an analytically tame groupoid, and suppose $p(z)\in \N[z]$ is a polynomial of degree at least 2 such that $p(\chi_a(\mathcal{G}))=\chi_a(\mathcal{G})$. Does this imply the existence of a nested equivalence $F:\mathcal{G}\to p_0(Z)+p_1(Z)\mathcal{G}+\cdots +p_m(Z)\mathcal{G}^m$ such that $P_F(z)=p(z)$?
\end{question}

\addcontentsline{toc}{section}{\numberline{}Bibliography}
\bibliographystyle{plain}
\bibliography{AGC3}

\end{document}